\documentclass[12pt]{amsart}

\usepackage[left=1.2in,right=1.2in, top=1.5in]{geometry}
\usepackage{amsmath}
\usepackage{amsthm}
\usepackage{amsfonts}
\usepackage{amssymb}
\usepackage{amsrefs}
\usepackage[colorlinks]{hyperref}
\usepackage{}
\allowdisplaybreaks

\hypersetup{
linkcolor=blue,         
citecolor=blue,        
urlcolor=blue
}

\newtheorem{theorem}{Theorem}[section]

\newtheorem{lemma}{Lemma}[section]

\theoremstyle{definition}
\newtheorem{definition}{Definition}[section]
\newtheorem{remark}{Remark}[section]

\newcommand{\SL}{\operatorname{SL}}
\newcommand{\GL}{\operatorname{GL}}

\newcommand{\Ad}{\operatorname{Ad}}

\newcommand{\Stab}{\operatorname{Stab}}

\begin{document}
\title[Orbit closures of sparse unipotent orbits]{Topological rigidity of closures of certain sparse unipotent orbits in finite-volume quotients of $\prod_{i=1}^k\SL_2(\mathbb R)$}
\author{Cheng Zheng}
\address{School of Mathematical Sciences, Shanghai Jiao Tong University, China}
\email{zheng.c@sjtu.edu.cn}


\thanks{The author acknowledges the support of NSFC grant number 12201398, and the support by Institute of Modern Analysis-A Frontier Research Center of Shanghai.}
\maketitle

\vspace{-0.1in}
\begin{abstract}
We give a simple proof about the topological rigidity of closures of certain sparse unipotent orbits in $G/\Gamma$ where $G=\prod_{i=1}^k\SL_2(\mathbb R)$ and $\Gamma$ is an irreducible lattice in $G$.
\end{abstract}

\section{Introduction}
Ratner's Theorem is one of the central theorems in homogeneous dynamics which concerns the measure rigidity of unipotent group actions, the topological rigidity of closures of unipotent orbits and the equidistribution of one-parameter unipotent orbits \cite{D,M1,R1,R2,MT}. There has been increasing attention recently to the extension of Ratner's theorem to sparse unipotent orbits, namely, to investigating the distribution of the sparse orbit $\{u(t)\cdot p:t\in\Omega\}$ in $G/\Gamma$ where $\{u(t)\}_{t\in\mathbb R}$ is a unipotent flow on a homogeneous space $G/\Gamma$ and the parameter $t$ is sampled in a discrete subset of $\mathbb R$, and it is listed as one of the open problems in homogeneous dynamics \cite{G,Sh}.  The equidistribution of sparse unipotent orbits (sparse equidistribution problem) has been studied in the setting where $G/\Gamma$ is a nilmanifold or a finite-volume quotient of $\SL_2(\mathbb R)$, and one can refer to \cite{SU, Z, M2, TV, V, B1, GT, K1, K2, L, LS1, LS2} for related results.

The aim of this paper is to prove the topological rigidity of closures of certain sparse unipotent orbits in $G/\Gamma$ where $G=\prod_{i=1}^k\SL_2(\mathbb R)$ and $\Gamma$ is an irreducible lattice in $G$. Recall that a natural number $x$ is said to be an $N$-almost prime for some $N\in\mathbb N$ if there are at most $N$ prime factors (counted with multiplicities) in the prime factorization of $x$. We denote by $\Omega(N)$ the set of $N$-almost primes in $\mathbb N$. Let $\{u_1(t)\}_{t\in\mathbb R}$ be a one-parameter unipotent subgroup of $\SL_2(\mathbb R)$ and $$u(t)=(u_1(t),\underbrace{e,\cdots,e}_{(k-1)\textup{ times}})\in G$$ where $e$ is the identity in $\SL_2(\mathbb R)$, and write $U_0=\{u_1(t)\}_{t\in\mathbb R},\;U=\{u(t)\}_{t\in\mathbb R}$. In this paper, we give simple proofs of the following two theorems which extend the results of \cite{Z24}.
\begin{theorem}\label{mthm1}
There exists $L\in\mathbb N$ such that for any $p\in G/\Gamma$, either the closure of the sparse $U$-orbit $\{u(n)\cdot p: n\in\Omega(L)\}$ is equal to $G/\Gamma$, or there exists an abelian unipotent subgroup $F$ of $G$ containing $U$ such that $F\cdot p$ is a compact orbit in $G/\Gamma$ isomorphic to a $k$-dimensional torus.
\end{theorem}

\begin{theorem}\label{mthm2}
There exists $\gamma_0>0$ such that for any $0<\gamma<\gamma_0$ and any $p\in G/\Gamma$, either the closure of the sparse $U$-orbit $\{u(n^{1+\gamma})\cdot p: n\in\mathbb N\}$ is equal to $G/\Gamma$, or there exists an abelian unipotent subgroup $F$ of $G$ containing $U$ such that $F\cdot p$ is a compact orbit in $G/\Gamma$ isomorphic to a $k$-dimensional torus.
\end{theorem}

\begin{remark}
Note that when $k=1$, Theorems~\ref{mthm1} and~\ref{mthm2} are exactly the main theorems of \cite{Z24}. So in the rest of the paper, we assume that $k\geq 2$.
\end{remark}

\section{Effective Equidistribution of $U$-orbits in $G/\Gamma$}
In this section, we prove the effective equidistribution of $U$-orbits in $G/\Gamma$. This topic has been studied well \cite{LG03,SU,LGJ16,TV,BV16,E,TJ18,S15,PT18,KM12,S21,BG23,KM96,DKL,BFG,EMV,M12}, and here we are mainly motivated by \cite{S,V}. We remark that significant progress has been made recently about effective equidistribution of one-parameter unipotent orbits for which one can refer to \cite{LMW23,CY24,LY22} for more details.

Let $\{a_1(t)\}_{t\in\mathbb R}$ be a one-parameter diagonalizable subgroup in $\SL_2(\mathbb R)$ such that $\{u_1(t)\}_{t\in\mathbb R}$ is the unstable horospherical subgroup of $\{a_1(t)\}_{t\in\mathbb R}$, and write $$a(t)=(a_1(t),\underbrace{e,\cdots,e}_{(k-1)\textup{ times}})\in G.$$ Then $U$ is the unstable horospherical subgroup of $\{a(t)\}_{t\in\mathbb R}$ in $G$. Let $\Ad:G\to\GL(\mathfrak g)$ be the adjoint representation of $G$ on the Lie algebra $\mathfrak g$ of $G$ induced by conjugation $x\mapsto g\cdot x\cdot g^{-1}$. Let $\mathfrak u$ be the Lie algebra of $U$. There exists a positive constant $\alpha>0$ such that $$\text{Ad}(a(t))x=e^{\alpha t}\cdot x,\quad \forall x\in\mathfrak u.$$ We denote by $\exp$ the exponential map from $\mathfrak g$ to $G$. Fix a norm on $\mathfrak g$ and let $d(\cdot,\cdot)$ be the induced distance on $G$. We denote by $B_r$ the open ball of radius $r>0$ around $e$ in $G$ and $\mu$ the $G$-invariant probability measure on $G/\Gamma$.

\begin{definition}
For any $x\in G/\Gamma$, we define the injectivity radius at $x$ by the largest number $\eta>0$ with the property that the map $$B_\eta\to B_\eta\cdot x\subset G/\Gamma,\quad g\mapsto g\cdot x$$ is injective. We denote the injectivity radius at $x$ by $\eta(x)$.
\end{definition}

We write $A\ll B$ if there exists a constant $C>0$ such that $A\leq CB$, and we will specify the constant $C>0$ in the context if necessary. If $A\ll B$ and $B\ll A$, we write $A\sim B$. For $f\in C^k(G/\Gamma)$ let $\|f\|_{p,k}$ be the Sobolev $L^p$-norm involving all the Lie derivatives of order $\leq k$ of $f$ (with respect to a basis of $\mathfrak g$). Following \cite[Lemma 9.5]{V}, we give a simple proof of the following

\begin{theorem}\label{equid}
There exist constants $a,b>0$ and $l\in\mathbb N$ such that for any $x\in G/\Gamma$ and $f\in C^\infty(G/\Gamma)$ with $\|f\|_{\infty,l}<\infty$, we have $$\left|\frac1{T}\int_0^Tf(u(s)\cdot x)ds-\int_{G/\Gamma} fd\mu\right|\ll\frac1{T^a\eta^b}\|f\|_{\infty,l}.$$ Here $\eta=\eta(a(-\ln T/\alpha)\cdot x)$ is the injectivity radius at $a(-\ln T/\alpha)\cdot x$, and the implicit constant depends only on $G/\Gamma$ and $U$.
\end{theorem}
\begin{remark}
When $k=1$, effective equidistribution of horocycle orbits has been studied in \cite{S,LG03,SU,V,LGJ16,TV}. In the case that $\Gamma$ is co-compact ($k\geq2$), Theorem~\ref{equid} is proved in \cite{TJ18}. The general case about the effective equidistribution of horospherical orbits has been studied in \cite{E} using tools from spectral theory. Here the proof of Theorem~\ref{equid} only uses a thickening technique of Margulis~\cite{M04} and the exponential mixing property of $\{a(t)\}_{t\in\mathbb R}$. Note that the exponents $a$ and $b$ are not optimized and may not be as explicit as those in \cite{S,E}. Nevertheless Theorem~\ref{equid} (especially the error term) suffices for our purpose to prove the topological rigidity in $G/\Gamma$ ($k\geq2$).
\end{remark}

We will need the following theorem in the proof of Theorem \ref{equid}

\begin{theorem}[Kleinbock and Margulis \cite{KM99}]\label{bthm21}
There exist $\kappa>0$ and $l\in\mathbb N$ such that for any $f,g\in C^\infty (G/\Gamma)$ with $\|f\|_{\infty,l},\;\|g\|_{\infty,l}<\infty$, we have $$\left|(a(t)\cdot f,g)-\int_{G/\Gamma}f\int_{G/\Gamma}g\right|\ll e^{-\kappa t}\|f\|_{\infty,l}\|g\|_{\infty,l}$$ $$\left|(u(t)\cdot f,g)-\int_{G/\Gamma}f\int_{G/\Gamma}g\right|\ll (1+|t|)^{-\kappa}\|f\|_{\infty,1}\|g\|_{\infty,1}.$$ Here $(a(t)\cdot f)(x)=f(a(t)\cdot x)$ is the left translation of $f$ by $a(t)$ and $(u(t)\cdot f)(x)=f(u(t)\cdot x)$ is the left translation of $f$ by $u(t)$.
\end{theorem}

\begin{lemma}\label{bl31}
Let $x\in G/\Gamma$ and $I(1)=\{u(s):0\leq s\leq 1\}$. Then for any $y\in I(1)\cdot x$, we have $\eta(y)\sim\eta(x).$ Here the implicit constant depends only on $G$ and $U$.
\end{lemma}
\begin{proof}
Let $y=u\cdot x$ for some $u\in I(1)$. By definition, the map $B_{\eta(x)}\to B_{\eta(x)}\cdot x$ is injective. This implies that the map $u\cdot B_{\eta(x)}\cdot u^{-1}\to u\cdot B_{\eta(x)}\cdot u^{-1}\cdot y$ is injective and hence $\eta(x)\ll\eta(y)$ for some constant depending only on $I(1)$. The proof of $\eta(y)\ll\eta(x)$ is similiar with the relation $x\in I(1)^{-1}\cdot y$. This completes the proof of the lemma.
\end{proof}

Now we fix a non-negative compactly supported smooth function $g(x)$ on $\mathbb R$ with integral one, and for any $n\in\mathbb N$, $\gamma>0$ and $\delta>0$ define
\begin{eqnarray*}
g_{\delta,n,\gamma}(u)=\frac1{\delta^n}\int_0^\gamma\int_0^\gamma\cdots\int_0^\gamma g\left(\frac{u_1-t_1}\delta\right) g\left(\frac{u_2-t_2}\delta\right)\dots g\left(\frac{u_n-t_n}\delta\right)dt_1dt_2\dots dt_n.
\end{eqnarray*}
It is clear that the function $g_{\delta,n,\gamma}$ satisfies the following properties:
\begin{enumerate}
\item $\int_{\mathbb R^n}g_{\delta,n,\gamma}(u)du=\gamma^n$.
\item $g_{\delta,n,\gamma}(u)$ is supported around the box $[0,\gamma]\times[0,\gamma]\times\dots\times[0,\gamma]$.
\item $\int_{\mathbb R^n} |g_{\delta,n,\gamma}(u)-\chi_{[0,\gamma]^n}(u)|du\ll\delta(\gamma+\delta)^{n-1}$.
\end{enumerate}
Here the implicit constant depends only on $g(x)$.

\begin{lemma}\label{bl33}
Let $y\in G/\Gamma$. Then there exist constants $a,b>0$ and $l\in\mathbb N$ such that for any $t>0$, $\gamma<\frac{\eta(y)}2$ and and $f\in C^\infty(G/\Gamma)$ with $\|f\|_{\infty,l}<\infty$ and $\int_{G/\Gamma} f d\mu=0$, we have
\begin{eqnarray*}
\left|\int_0^\gamma f(a(t)\cdot u(s)\cdot y)ds\right|\ll\frac1{e^{at}\gamma^b}\|f\|_{\infty,l}.
\end{eqnarray*}
The implicit constant depends only on $G/\Gamma$ and $U$.
\end{lemma}
\begin{proof}
Let $U^-$ be the stable horospherical subgroup of $\{a(t)\}_{t\in\mathbb R}$ and $Z=Z(\{a(t)\}_{t\in\mathbb R})$ the centralizer of $\{a(t)\}_{t\in\mathbb R}$ in $G$. Denote by $\mathfrak u^-$ and $\mathfrak z$ the Lie algebra of $U^-$ and $Z$ respectively. Then we have $\mathfrak u\oplus\mathfrak z\oplus\mathfrak u^{-}=\mathfrak g.$ Let $n=\dim U=\dim U^-=1$, $m=\dim Z$ and $l\in\mathbb N$ as in Theorem~\ref{bthm21}. Let $\{X_1,X_2,\cdots,X_m\}$ be a basis of $\mathfrak z$ and $\mathfrak u^-=\mathbb R\cdot Y$ for some $Y\in\mathfrak g$. By the argument in \cite[Lemma 9.5]{V}, for any $\gamma<\eta(y)/2$, we have
\begin{eqnarray*}
&{}&\int_0^\gamma f(a(t)\cdot u(s)\cdot y)ds\\
&=&\int_{\mathbb R} f(a(t)\cdot u(s)\cdot y)g_{\delta,n,\gamma}(s)ds+O(\|f\|_{\infty,l})\delta(\gamma+\delta)^{n-1}\\
&=&\frac1{\delta^m\gamma^n}\iiint_{\mathbb R\times\mathbb R^m\times\mathbb R}f(a(t)\cdot u(s)\cdot y)g_{\delta,n,\gamma}(u)g_{\delta,m,\delta}(z)g_{\delta,n,\gamma}(v)dsdzdv+O(\|f\|_{\infty,l})\delta(\gamma+\delta)^{n-1}\\
&=&\frac1{\delta^m\gamma^n}\iiint_{\mathbb R\times \mathbb R^m\times\mathbb R}f(a(t)\cdot\exp(v\cdot Y)\cdot\exp(z_1X_1+\cdots z_m\cdot X_m)\cdot u(s)\cdot y)\\
&{}&\hspace{0.5cm}g_{\delta,n,\gamma}(u)g_{\delta,m,\delta}(z)g_{\delta,n,\gamma}(v)dudzdv+O(\|f\|_{\infty,l})(\delta(\gamma+\delta)^{n-1}+\gamma^n\max\{\delta,\gamma/e^{\alpha t}\})\\
&=&\frac1{\delta^m\gamma^n}\int_{G/\Gamma} f(a(t)\cdot x)g_{\delta,y}(x)d\mu(x)+O(\|f\|_{\infty,l})(\delta(\gamma+\delta)^{n-1}+\gamma^n\max\{\delta,\gamma/e^{\alpha t}\})\\
&=&\frac1{\delta^m\gamma^n}(a(t)\cdot f,g_{\delta,y})+O(\|f\|_{\infty,l})(\delta(\gamma+\delta)^{n-1}+\gamma^n\max\{\delta,\gamma/e^{\alpha t}\})
\end{eqnarray*}
where $g_{\delta,y}$ is a function supported on the ball of radius $\eta(y)$ at $y$ in $G/\Gamma$. Note that the injectivity radius function $\eta(y)$ is bounded above by a constant depending only on $G/\Gamma$. By the definition of Lie derivatives, we can compute $\|g_{\delta,y}\|_{\infty,l}$ and there exists a constant $p>0$ such that $\|g_{\delta,y}\|_{\infty,l}\ll1/\delta^p.$ Therefore, by Theorem~\ref{bthm21}, we have
\begin{eqnarray*}
\left|\int_0^\gamma f(a(t)\cdot u(s)\cdot y)ds\right|\ll\frac1{\delta^m\gamma^n}\frac1{e^{\kappa t}\delta^p}\|f\|_{\infty,l}+(\delta(\gamma+\delta)^{n-1}+\gamma^n\max\{\delta,\gamma/e^{q t}\})\|f\|_{\infty,l}.
\end{eqnarray*}
Let $\delta=\gamma e^{-\epsilon t}<\gamma$ for some small $\epsilon>0$ and this completes the proof of the lemma.
\end{proof}

\begin{proof}[Proof of Theorem \ref{equid}]
Without loss of generality, assume that $\int fd\mu=0$. By Lemma \ref{bl31}, we can find $\gamma>0$ with the following properties
\begin{enumerate}
\item the interval $I(1)$ can be divided into small subintervals $\{B_j\}$ such that for each $j$, there exists $y_j\in I(1)$ such that $B_j=I(\gamma)\cdot y_j$.
\item For each $j$, we have $\gamma<\eta(y_j\cdot a(-\ln T/\alpha)\cdot x)/2$.
\item $\gamma\sim\eta(a(-\ln T/\alpha)\cdot x)$ and the implicit constant depends only on $G/\Gamma$ and $U$.
\end{enumerate}
In fact, we can pick such $\gamma$ by first taking the infimum of $\{\eta(y\cdot a(-\ln T/\alpha)\cdot x)/2|y\in I(1)\}$ and then modifying it so that $1/\gamma$ is an integer. Note that the number of the intervals $B_j$ is $1/\gamma$. Now by Lemma \ref{bl33} we have

\begin{eqnarray*}
&{}&\left|\frac1{T}\int_0^Tf(u(s)\cdot x)ds\right|=\left|\int_0^1f(a(\ln T/\alpha)\cdot u(s)\cdot a(-\ln T/\alpha)\cdot x)ds\right|\\
&\leq&\sum_j\left|\int_{B_j}f(a(\ln T/\alpha)\cdot u\cdot a(-\ln T/\alpha)\cdot x)du\right|\\
&=&\sum_j\left|\int_0^\gamma f((a(\ln T/\alpha)\cdot u(s))\cdot(y_ja(-\ln T/\alpha)x))ds\right|\\
&\ll&\frac1{\gamma}\cdot\frac1{T^{a/\alpha}\gamma^b}\|f\|_{\infty,l}\ll\frac1{T^{a/\alpha}\eta(a(-\ln T/\alpha)\cdot x)^{b+1}}\|f\|_{\infty,l}.
\end{eqnarray*}
This completes the proof of Theorem \ref{equid}.
\end{proof}

The following theorem is crucial in the proofs of Theorems~\ref{mthm1} and~\ref{mthm2}, the proof of which is similar to \cite[Theorem 3.1]{V} and \cite[Theorem 1.2]{Z}. For self-containedness, we give a sketch of the proof.
\begin{theorem}\label{dequid}
Let $a,b>0$ and $l\in\mathbb N$ be as in Theorem~\ref{equid}. Let $T>K\geq 1$ and $f\in C^\infty(G/\Gamma)$ satisfying $\int_{G/\Gamma}fd\mu=0$ and $\|f\|_{\infty,l}<\infty$. Suppose that $q\in G/\Gamma$ satisfies $r=T^{a}\eta(a(-\ln T/\alpha)\cdot q)^b\geq1$. Then there exists $\beta>0$ such that $$\left|\frac1{T/K}\sum\limits_{0\leq Kj<T,j\in\mathbb N}f(u(Kj)\cdot q)\right|\ll\frac{K^\frac12}{r^\beta} \|f\|_{\infty,l}.$$ The implicit constant depends only on $G/\Gamma$ and $U$.
\end{theorem}
\begin{proof}[Sketch of the proof of Theorem~\ref{dequid}]
We mainly follow the proof of \cite[Theorem 1.2]{Z}, which is explained in \cite[\S5]{Z}. Note that here $r=T^{a}\eta(a(-\ln T/\alpha)\cdot q)^b\geq1$ and $r\ll T^a$ where the implicit constant depends only on $G/\Gamma$. Now we can run the argument in \cite[\S5]{Z} with 
$T\cdot e^{-\operatorname{dist}(g_{\log T}(q))}$ replaced by $T^{a}\eta(a(-\ln T/\alpha)\cdot q)^b$ and \cite[Theorem 2.2]{Z} replaced by Theorem~\ref{equid}, and conclude that there exists $\beta>0$ such that $$\left|\frac1{T/K}\sum\limits_{0\leq Kj<T,j\in\mathbb N}f(u(Kj)\cdot q)\right|\ll\left(\frac KT+\frac {K^{\frac12}}{r^\beta}\right)\|f\|_{\infty,l}.$$ Note that $K<T$ and $1\leq r\ll T^a$. This completes the proof of the theorem.
\end{proof}

\section{Proof of Theorem~\ref{mthm1} for points with non-divergent $a(-t)$-orbits}\label{sieve}
In this section, we prove Theorem~\ref{mthm1} in the case that the $a(-t)$-orbit of $p$ does not diverge. We will need the following Jurkat-Richert theorem about linear sieve. For any $x\in\mathbb R$, we denote by $[x]$ the largest integer $\leq x$. 
\begin{theorem}{\cite[Theorem 9.7]{MN}}\label{th31}
Let $A=\{a(n)\}_{n\in\mathbb N}$ be a sequence of non-negative numbers such that $$|A|=\sum_{n=1}^\infty a(n)<\infty.$$ Let $\mathcal P$ be a set of prime numbers and for $z\geq2$, let $$P(z)=\prod_{p\in\mathcal P,p<z}p.$$ Let $$S(A,\mathcal P,z)=\sum_{n=1,(n,P(z))=1}^\infty a(n).$$ For every $n\geq1$, let $g_n(d)$ be a multiplicative function such that $$0\leq g_n(p)<1,\textup{ for all }p\in\mathcal P.$$ Define $r(d)$ by $$|A_d|=\sum_{n=1,d|n}^\infty a(n)=\sum_{n=1}^\infty a(n)g_n(d)+r(d).$$ Let $\mathcal Q$ be a finite subset of $\mathcal P$, and let $Q$ be the product of the primes in $\mathcal Q$. Suppose that, for some $\epsilon$ satisfying $0<\epsilon<1/200$, the inequality $$\prod_{p\in\mathcal P\setminus\mathcal Q, u\leq p<z}(1-g_n(p))^{-1}<(1+\epsilon)\frac{\log z}{\log u}$$ holds for all $n$ and $1<u<z$. Then for any $D\geq z$ there is the upper bound $$S(A,\mathcal P, z)<(F_0(s)+\epsilon e^{14-s})X+R$$ and for any $D\geq z^2$ there is the lower bound $$S(A,\mathcal P,z)>(f_0(s)-\epsilon e^{14-s})X-R$$ for some functions $F_0(s)$ and $f_0(s)$ where $s=\frac{\log D}{\log z}$, $F_0(s)=1+O(e^{-s})$, $f_0(s)=1-O(e^{-s})$, $$X=\sum_{n=1}^\infty a(n)\prod_{p\in P(z)}(1-g_n(p))$$ and the remainder term is $$R=\sum_{d|P(z),d<DQ}|r(d)|.$$ If there is a multiplicative function $g(d)$ such that $g_n(d)=g(d)$ for all $n$, then $X=V(z)|A|$, where $$V(z)=\prod_{p|P(z)}(1-g(p)).$$
\end{theorem}
\begin{remark}
The explicit expressions of the functions $F_0(s)$ and $f_0(s)$ can be found in \cite[Theorem 9.4]{MN}.
\end{remark}

The following theorem will be used to verify an assumption in Theorem~\ref{th31}.
\begin{theorem}{\cite[Theorem 6.9]{MN}}\label{th32}
For any $\epsilon>0$, there exists a number $\tilde u(\epsilon)>0$ such that $$\prod_{u\leq p<z}\left(1-\frac1p\right)^{-1}<(1+\epsilon/3)\frac{\log z}{\log u}$$ for any $\tilde u(\epsilon)\leq u<z$.
\end{theorem}

\begin{proof}[Proof of Theorem~\ref{mthm1} assuming that $\{a(-t)\cdot p\}_{t\geq0}$ does not diverge]
The proof is similar to \cite[\S5]{Z24}. We prove that there exists $L\in\mathbb N$ such that the orbit $\{u(n)\cdot p: n\in\Omega(L)\}$ is dense in $G/\Gamma$. Since $\{a(-t)\cdot p\}_{t\geq0}$ returns to a compact subset $\mathcal K$ in $G/\Gamma$ infinitely often, there exists a divergent sequence of positive integers $T_i\to\infty$ such that $a(-\ln T_i/\alpha)\cdot p\in\mathcal K$.

Let $f\neq0$ be a non-negative compactly supported smooth function on $G/\Gamma$ and let $N=T_i$ for sufficiently large $i\in\mathbb N$. Define 
$$\begin{cases} a(n)=f(u(n)\cdot p) & 0\leq n\leq N \\ a(n)=0 & n>N\end{cases}.$$ We take the set $\mathcal P$ in Theorem~\ref{th31} to be the set of all prime numbers. Then using the notation in Theorem~\ref{th31} and the result in Theorem~\ref{dequid}, we have 
\begin{align*}
|A_d|=&\sum_{1\leq dj\leq N}f(u(jd)\cdot p)=\frac Nd\int_{G/\Gamma}fd\mu+O\left(\left(\frac Nd\right)\frac{d^\frac12}{r^{\beta}}\|f\|_{\infty,l}\right)+O\left(\|f\|_{\infty,l}\right)\\
=&\sum_{1\leq n\leq N}\frac1d a(n)+O\left(\left(\left(\frac Nd\right)\frac{d^\frac12}{r^{\beta}}+1\right)\|f\|_{\infty,l}\right).
\end{align*}
Note that here $r=T_i^{a}\cdot\eta(a(-\ln T_i/\alpha)\cdot p)^b\sim T_i^a=N^{a}$ where the implicit constant depends only on $\mathcal K$ and $b$. In Theorem~\ref{th31}, we take $g_n(d)=g(d)=1/d$ and $$|r(d)|\ll\left(\left(\frac Nd\right)\frac{d^\frac12}{r^{\beta}}+1\right)\|f\|_{\infty,l}.$$ 

Now set $z=N^\alpha$ for some $\alpha>0$ which we will determine later. Let $s>100$ be a sufficiently large number so that $f_0(s)>0.1$, where $f_0(s)$ is defined as in Theorem~\ref{th31}. Fix $\epsilon\in(0,1/200)$ and take $\tilde u(3\epsilon)>0$ as in Theorem~\ref{th32}. Let $\mathcal Q$ be the subset of primes $<\tilde u(3\epsilon)$. Then by Theorem~\ref{th31} (and with the help of Theorem~\ref{th32}), we have $$S(A,\mathcal P,z)>0.01V(z)|A|-R,$$ where $D=z^s=N^{\alpha s}$. By Theorem~\ref{th32} and Theorem~\ref{dequid} we can compute that $$|A|/N\sim \int_{G/\Gamma}fd\mu\textup{ and }V(z)\gg1/\log N\textup{ for }N=T_i \;(i\textup{ sufficiently large}).$$ If $\alpha$ is chosen to be sufficiently small, then $|R|\ll N^{1-\delta}$ for some $\delta>0$ and $S(A,\mathcal P,z)>0$ if $N=T_i$ is sufficiently large. Note that in the formula of $S(A,\mathcal P,z)$, any $n\leq N$ coprime to $P(z)$ has at most $[1/\alpha]+1$ prime factors. If we take $L=[1/\alpha]+1$ and $\Omega(L)$ the set of $L$-almost primes, then we obtain $$\sum_{n\in\Omega(L)} f(u(n)\cdot p)\geq S(A,\mathcal P,z)>0.$$ Since $f\neq0$ is any non-negative compactly supported smooth function, we conclude that the orbit $\{u(n)\cdot p:n\in\Omega(L)\}$ is dense in $G/\Gamma$. This completes the proof of Theorem~\ref{mthm1} when the orbit $\{a(-t)\cdot p\}_{t\geq0}$ does not diverge.
\end{proof}

\section{Proof of Theorem~\ref{mthm1} for points with divergent $a(-t)$-orbits}
Now we consider the points $p$ for which the orbit $\{a(-t)\cdot p\}_{t\in\mathbb R}$ diverges. In this case, $\Gamma$ is not a co-compact lattice in $G$, and we will show that the orbit $\{u(t)\cdot p\}_{t\in\mathbb R}$ is inside a $k$-dimensional torus.

By Margulis arithmeticity theorem~\cite{M84,M91}, the irreducible lattice $\Gamma$ is arithmetic ($k\geq 2$). Moreover, there exist a number field $k/\mathbb Q$ with completion $k_\infty=\mathbb R$ or $\mathbb C$, a connected simple subgroup $H$ of $\SL_n(k_\infty)$ defined over $k$, and a continuous surjection $$\phi:\prod_{\sigma\in S^\infty}H^\sigma\to G$$ with compact kernel such that $\phi(\Delta(H_{\mathcal O}))$ is commensurable with $\Gamma$, where $S_\infty=\{\sigma_1,\dots,\sigma_l\}$ $(l\in\mathbb N)$ is the set of all inequivalent embeddings of $k$ into $\mathbb C$, $\mathcal O$ is the ring of integers in $k$ and $\Delta(H_{\mathcal O})$ is the diagonal embedding of $\mathcal O$-points $H_{\mathcal O}$ of $H$ into $\prod_{\sigma\in S^\infty}H^\sigma$ by $S_\infty$ (see e.g. \cite[Corollary 5.5.15]{MDW}). Note that $G=\prod_{i=1}^k\SL_2(\mathbb R)$ and $G/\Gamma$ is not compact. This implies that $k_\infty=\mathbb R$ and $\prod_{\sigma\in S^\infty}H^\sigma$ is isogenous to $G$ without compact factors. So without loss of generality, we may assume that $G=\prod_{\sigma\in S^\infty}H^\sigma$, $\Gamma$ is a subgroup of finite-index in the arithmetic group $\Delta(H_{\mathcal O})$, and $\{u(t)\}_{t\in\mathbb R}$ and $\{a(t)\}_{t\in\mathbb R}$ are one-parameter subgroups in $H^{\sigma_1}$. We denote by $\pi_\sigma: G\to H^\sigma$ the projection from $G$ to $H^\sigma$ ($\forall\sigma\in S_\infty$), $\mathfrak h$ the Lie algebra of $H^{\sigma_1}$ and $\Stab(p)$ the stablizer of $p$ in $G$ ($\forall p\in G/\Gamma$).

The following lemma follows from \cite[Lemma 3.1]{BZ17}.
\begin{lemma}\label{l41}
Any connected nilpotent subgroup in $\SL_2(\mathbb R)$ containing a nontrivial unipotent element is a conjugate of $U_0$.
\end{lemma}

In the following, we will fix a sufficiently small constant $\rho_0>0$ such that the open ball $B_{\rho_0}$ of radius $\rho_0>0$ in $G$ satisfies the following properties
\begin{enumerate}
\item $B_{\rho_0}$ is a Zassenhaus neighborhood of $e$ in $G$, which is a neighborhood of $e$ such that for any discrete subgroup $\Delta\subset G$, there exists a connected nilpotent subgroup $F\subset G$ such that $\Delta\cap B_{\rho_0}\subset F$. Note that the existence of a Zassenhaus neighborhood is proved in e.g. \cite[Theorem 8.16]{R72}.
\item $B_{\rho_0}$ is a neighborhood of $e$ satisfying \cite[Corollary 11.18]{R72}.
\end{enumerate}

Since $\{a(-t)\cdot p\}_{t\geq0}$ diverges, by \cite[Theorem 1.12]{R72}, there exists $t_0>0$ such that for any $t\geq t_0$ one can find $\gamma_t\in\Stab(p)\setminus\{e\}$ such that $a(-t)\cdot\gamma_t\cdot a(t)\in B_{\rho_0}$ and $a(-t)\cdot \gamma_t\cdot a(t)\to e$ as $t\to\infty$. By the choice of $B_{\rho_0}$, $\gamma_t$ is unipotent. By the assumption on $\Gamma$, we can write $\gamma_t=(\gamma_t^\sigma)_{\sigma\in S_\infty}$ where each $\gamma_t^\sigma$ is unipotent in $H^\sigma$. Now let $F_t$ be a nilpotent subgroup in $G$ such that $$B_{\rho_0}\cap a(-t)\cdot\Stab(p)\cdot a(t)\subset F_t$$ and we may assume that $F_t=\prod_{\sigma\in S_\infty}\pi_{\sigma}(F_t)$. Since $\gamma_t^\sigma\in\pi_{\sigma}(F_t)$, by Lemma~\ref{l41} and isogeny, it implies that $\pi_{\sigma}(F_t)$ is a one-parameter unipotent subgroup in $H^\sigma$ $(\forall\sigma\in S_\infty)$, and $F_t$ is a $k$-dimensional abelian unipotent subgroup in $G$. Note that by the assumption on $\Gamma<\Delta(H_{\mathcal O})$, $F_t\cap a(-t)\cdot\Stab(p)\cdot a(t)$ is a lattice in $F_t$.

\begin{lemma}\label{l42}
Let $p\in G/\Gamma$ and suppose that $\{a(-t)\cdot p\}_{t\geq0}$ diverges. Then there exists a $k$-dimensional abelian unipotent subgroup $F\subset G$ such that $F\cap\Stab(p)$ is a lattice in $F$ and $U\subset F$.
\end{lemma}
\begin{proof}
Let $\mathcal F$ be the collection of all $k$-dimensional abelian unipotent subgroups $F$ of $G$ such that
\begin{enumerate}
\item $F=\prod_{\sigma\in S_\infty}\pi_\sigma(F)$.
\item $F\cap\Stab(p)$ is a lattice in $F$.
\end{enumerate}
For any $F\in\mathcal F$, define $$I_F=\{t\in\mathbb R: a(-t)(F\cap\Stab(p))a(t)\cap B_{\rho_0}\neq\{e\},t\geq t_0\}.$$ Then by the discussion above, we have $$[t_0,+\infty)=\bigcup_{F\in\mathcal F} I_{F}.$$ We claim that the subsets $I_F$ are pairwise disjoint. Suppose that there exist $F_1,F_2\in\mathcal F$ such that $I_{F_1}\cap I_{F_2}\neq\emptyset.$ Let $t\in I_{F_1}\cap I_{F_2}$. Then we have $$a(-t)(F_1\cap\Stab(p))a(t)\cap B_{\rho_0}\neq\{e\},\quad a(-t)(F_2\cap\Stab(p))a(t)\cap B_{\rho_0}\neq\{e\}.$$ Let $\gamma\in a(-t)\cdot\Stab(p)\cdot a(t)\cap B_{\rho_0}\setminus\{e\}$. Then by the discussion above, $\gamma=(\gamma^\sigma)_{\sigma\in S_\infty}$ uniquely determines the unipotent subgroup $F_t$ such that $$a(-t)\cdot\Stab(p)\cdot a(t)\cap B_{\rho_0}\subset F_t.$$ This implies that $$a(-t)\cdot F_1\cdot a(t)=a(-t)\cdot F_2\cdot a(t)=F_t$$ and $F_1=F_2$.

Note that $I_F$ are relatively open subsets in $[t_0,\infty)$, and each $I_F$ is a disjoint union of intervals in $[t_0,\infty)$. By the connectedness of $[t_0,\infty)$, this implies that there is a subgroup $F_0\in\mathcal F$ such that $[t_0,\infty)=I_{F_0}$ and we have $$a(-t)\cdot F_0\cdot a(t)=F_t.$$  Now we claim that $U\subset F_0$. It suffices to prove that $\pi_{\sigma_1}(F_0)$ is the stable horospherical subgroup $U$ of $\{a(-t)\}_{t\in\mathbb R}$ in $H^{\sigma_1}$. Suppose that $\pi_{\sigma_1}(F_0)$ is not the stable horospherical subgroup of $\{a(-t)\}_{t\in\mathbb R}$ in $H^{\sigma_1}$. Let $v\in\mathfrak h$ be the nilpotent element such that $$\exp(\mathbb R\cdot v)=\pi_{\sigma_1}(F_0).$$ Note that $\mathfrak h$ is isomorphic to $\mathfrak{sl}_2(\mathbb R)$ and one can find a basis $\{X,Y,Z\}$ of $\mathfrak h$ such that $\exp(\mathbb R\cdot X)$ is the stable horospherical subgroup of $\{a(-t)\}_{t\in\mathbb R}$, $\exp(\mathbb R\cdot Y)$ is the unstable horospherical subgroup of $\{a(-t)\}_{t\in\mathbb R}$, $Z$ commutes with $\{a(t)\}_{t\in\mathbb R}$ and $[Z,X]=2X,\; [Z,Y]=-2Y,\; [X,Y]=Z$. Then we can write $v$ as $$v=aX+bY+cZ$$ where $b\neq 0$. One can compute that $$\Ad(a(-t))(v)=a\cdot e^{-\alpha t}\cdot X+b\cdot e^{\alpha t}\cdot Y+c\cdot Z$$ and there exists a sufficiently large number $t_1>0$ such that for any $t>t_1$ $$\|\Ad(a(-t))v\|\geq\|v\|.$$ This implies that for any $u\in\mathbb R\cdot v$, $$\|\Ad(a(-t))u\|\geq\|u\|.$$ From the inequalities above, one can deduce that for any $\gamma\in\Stab(p)\cap F_0\setminus\{e\}$ and any $t\geq t_1$, the element $\Ad(a(-t))\gamma$ is bounded away from the identity $e$. This contradicts the condition that $[t_0,+\infty)=I_{F_0}$ and $$\inf_{v\in a(-t)\cdot(\Stab(p)\cap F_0)\cdot a(t)\setminus\{e\}}d(v,e)\to0\quad(t\to\infty).$$ This completes the proof of the lemma.
\end{proof}

By Lemma~\ref{l42}, if $\{a(-t)\cdot p\}_{t\geq0}$ diverges, then there exists a $k$-dimensional abelian unipotent subgroup $F$ such that $F\cap\Stab(p)$ is a lattice in $F$ and $U\subset F$. Therefore, the orbit $U\cdot p$ is inside the compact unipotent $F$-orbit $F\cdot p$ which is a $k$-dimensional torus. This completes the proof of Theorem~\ref{mthm1}.

\section{Proof of Theorem~\ref{mthm2}}
In this section, we give a brief explanation about the proof of Theorem~\ref{mthm2}, which is similar to that of Theorem~\ref{mthm1}. In the following, for any $\gamma>0$, we define a map $$\phi(x)=a_{-\log x/2\alpha}\cdot u(x^{1+\gamma})\in G\quad(x>0).$$

\begin{lemma}\label{l51}
Let $p\in G/\Gamma$ and suppose that $\{\phi(t)\cdot p\}_{t>0}$ diverges. Then there exists a $k$-dimensional abelian unipotent subgroup $F\subset G$ such that $F\cap\Stab(p)$ is a lattice in $F$ and $U\subset F$.
\end{lemma}
\begin{proof}
Let $\mathcal F$ be the collection of all $k$-dimensional abelian unipotent subgroups $F$ of $G$ such that
\begin{enumerate}
\item $F=\prod_{\sigma\in S_\infty}\pi_\sigma(F)$.
\item $F\cap\Stab(p)$ is a lattice in $F$.
\end{enumerate}
For any $F\in\mathcal F$, define $$I_F=\{t\in\mathbb R: \phi(t)\cdot(F\cap\Stab(p))\cdot\phi(t)^{-1}\cap B_{\rho_0}\neq\{e\},t\geq t_0\}.$$ Using the same argument as in Lemma~\ref{l42}, one can deduce that there is a subgroup $F_0\in\mathcal F$ such that $[t_0,\infty)=I_{F_0}.$ Now we claim that $U\subset F_0$. It suffices to prove that $\pi_{\sigma_1}(F_0)$ is the stable horospherical subgroup $U$ of $\{a(-t)\}_{t\in\mathbb R}$ in $H^{\sigma_1}$. Suppose that $\pi_{\sigma_1}(F_0)$ is not the stable horospherical subgroup of $\{a(-t)\}_{t\in\mathbb R}$ in $H^{\sigma_1}$. Let $v\in\mathfrak h$ be the nilpotent element such that $\exp(\mathbb R\cdot v)=\pi_{\sigma_1}(F_0).$ Note that $\mathfrak h$ is isomorphic to $\mathfrak{sl}_2(\mathbb R)$ and one can find a basis $\{X,Y,Z\}$ of $\mathfrak h$ such that $\exp(\mathbb R\cdot X)$ is the stable horospherical subgroup of $\{a(-t)\}$, $\exp(\mathbb R\cdot Y)$ is the unstable horospherical subgroup of $\{a(-t)\}$, $Z$ commutes with $\{a(t)\}_{t\in\mathbb R}$ and $[Z,X]=2X,\; [Z,Y]=-2Y,\; [X,Y]=Z$. Then we can write $v$ as $$v=aX+bY+cZ$$ where $b\neq 0$. One can compute that in terms of the basis $\{X,Y,Z\}$, the $Y$-coefficient of the element $\Ad(\phi(t))(v)\in\mathfrak h$ is equal to $b\cdot t^{\frac12}$, and there exists a sufficiently large number $t_1>0$ such that for any $t>t_1$ $$\|\Ad(\phi(t))v\|\geq\|v\|.$$ This implies that for any $u\in\mathbb R\cdot v$, we have $$\|\Ad(\phi(t))u\|\geq\|u\|.$$ From the inequality above, one can deduce that for any $\gamma\in\Stab(p)\cap F_0\setminus\{e\}$ and any $t\geq t_1$, the element $\Ad(\phi(t))\gamma$ is bounded away from the identity $e$, which contradicts the condition that $[t_0,+\infty)=I_{F_0}$ and $$\inf_{v\in\phi(t)\cdot(\Stab(p)\cap F_0)\cdot\phi(t)^{-1}\setminus\{e\}}d(v,e)\to0\quad(t\to\infty).$$ This completes the proof of the lemma.
\end{proof}

\begin{proof}[Proof of Theorem~\ref{mthm2}]
We may assume that $f\in C^\infty(G/\Gamma)$ with $\|f\|_{\infty,l}<\infty$ and $\int_{G/\Gamma}fd\mu=0.$ We will find $\gamma_0>0$ such that for any $0<\gamma<\gamma_0$ Theorem~\ref{mthm2} holds. Note that by Taylor expansion, for any $M\in\mathbb N$ and $k\in\mathbb N,$ $$(M+k)^{1+\gamma}=M^{1+\gamma}+(1+\gamma)M^\gamma k+O(M^{\gamma-1}k^2).$$ Therefore, if $M$ is sufficiently large and $\gamma<1/2$, the sequence $$P_M=\left\{(M+k)^{1+\gamma}\Big|\;0\leq k\leq\frac1{1+\gamma} M^{\frac12-\gamma}\right\}$$ is approximated by the arithmetic progression $$\tilde P_M=\left\{M^{1+\gamma}+(1+\gamma)M^\gamma k\Big|\;0\leq k\leq\frac1{1+\gamma}M^{\frac12-\gamma}\right\}$$ since $$O(M^{\gamma-1}k^2)\leq O(M^{\gamma-1}(M^{\frac12-\gamma})^2)=O(M^{-\gamma})\to0\textup{ as }M\to\infty.$$

Suppose that $\{\phi(n)\cdot p\}_{n\in\mathbb N}$ does not diverge where $\phi(n)=a_{-\log n/2\alpha}\cdot u(n^{1+\gamma}).$ Then there exists a sequence of positive integers $M_i$ such that $\phi(M_i)\cdot p$ belongs to a compact subset $\mathcal K\subset G/\Gamma$. Let $T_i=M_i^{\frac12}$ and $q_i=u(M_i^{1+\gamma})\cdot p$. We compute that $a(-\ln T_i/\alpha)\cdot q_i=\phi(M_i)\cdot p\in\mathcal K$, which implies that $\eta(a(-\ln T_i/\alpha)\cdot q_i)$ is bounded below by a constant depending only on $\mathcal K$.

By Theorem \ref{dequid} with $T=T_i=M_i^{1/2}$, $K=(1+\gamma)M_i^\gamma$ and $q=q_i=u(M_i^{1+\gamma})\cdot p$ we have
\begin{eqnarray*}
\left|\frac1{|\tilde P_{M_i}|}\sum\limits_{n\in\tilde P_{M_i}}f(u(n)\cdot p)\right|\ll\frac {((1+\gamma)M_i^\gamma)^\frac12}{r^\beta}\|f\|_{\infty,l}.
\end{eqnarray*}
Since $M_i\to\infty$ and $r=T_i^{a}\cdot\eta(a(-\ln T_i/\alpha)\cdot q_i)^b\sim T_i^a$, as long as $\gamma<a\cdot\beta$, we have $$\lim_{i\to\infty}\frac1{|\tilde P_{M_i}|}\sum\limits_{n\in \tilde P_{M_i}}f(u(n)\cdot p)=0\textup{ and }\lim_{i\to\infty}\frac1{|P_{M_i}|}\sum\limits_{n\in P_{M_i}}f(u(n)\cdot p)=0.$$ Consequently, for any $f\in C_c^\infty(G/\Gamma)$, we have $$\lim_{i\to\infty}\frac1{|P_{M_i}|}\sum\limits_{n\in P_{M_i}}f(u(n)\cdot p)=\int_{G/\Gamma} fd\mu.$$ This proves that $\{u(n^{1+\gamma})\cdot p\}$ is dense in $G/\Gamma$ if $\gamma<\gamma_0:=a\cdot\beta$.

Now assume that $\{\phi(n)\cdot p\}_{n\in\mathbb N}$ diverges. One can compute that for any $\delta\in(0,1)$ and $n\geq1$, the element $\phi(n+\delta)\cdot\phi(n)^{-1}$ lies in a compact subset in $H^{\sigma_1}$. It implies that $\{\phi(x)\cdot p\}_{x>0}$ diverges as well, and hence by Lemma~\ref{l51}, there is a $k$-dimensional abelian unipotent subgroup $F$ containing $U$ such that $F\cap\Stab(p)$ is a lattice in $F$. This completes the proof of Theorem~\ref{mthm2}.
\end{proof}

\vspace{0.2in}
\noindent\textbf{Acknowledgements.} The author would like to thank Professor Andreas Str\"ombergsson and Samuel Edwards for valuable discussions about Theorem~\ref{equid} regarding the effective equidistribution of horospherical orbits in homogeneous spaces.

\end{document}